\documentclass[10pt, a4paper]{amsart}
\usepackage{amsfonts}
\usepackage{amsthm}
\usepackage{amsmath}
\usepackage{latexsym}
\usepackage{amssymb}
\usepackage{graphicx}
\usepackage{url}

\newtheorem{lemma}{Lemma}[section]
\newtheorem{proposition}[lemma]{Proposition}

\theoremstyle{remark}
\newtheorem{remark}[lemma]{Remark}

\numberwithin{equation}{section}

\newcommand{\Q}{\mathbb{Q}}
\newcommand{\C}{\mathbb{C}}
\newcommand{\Z}{\mathbb{Z}}
\newcommand{\R}{\mathbb{R}}

\newcommand{\g}{\mathop{\mathfrak{g}}}
\newcommand{\h}{\mathop{\mathfrak{h}}}
\newcommand{\p}{\mathop{\mathfrak{p}}}
\newcommand{\z}{\mathop{\mathfrak{z}}}

\newcommand{\Ad}{\mathop{\mathrm{Ad}}}

\newcommand{\ad}{\mathrm{\mathop{ad}}}

\newcommand{\Tr}{\mathrm{\mathop{Tr}}}
\newcommand{\mf}[1]{\mathfrak{#1}}
\newcommand{\ssl}{\mathfrak{\mathop{sl}}}
\newcommand{\myl}{\mathfrak{l}}
\newcommand{\hh}{\mathfrak{h}}

\newcommand{\GAP}{{\sf GAP}}

\begin{document}

\title{Exploring Lie theory with {\sf GAP}}
\author{Willem A. de Graaf}
\address{
Dipartimento di Matematica\\
Universit\`{a} di Trento\\
Italy}
\email{willem.degraaf@unitn.it}
\thanks{The author was partially supported by an Australian Research Council
grant, identifier DP190100317.}



\begin{abstract}
We illustrate the Lie theoretic capabilities of the computational algebra
system {\GAP}4 by reporting on results on nilpotent orbits of simple Lie
algebras that have been obtained using computations in that system.
Concerning reachable elements in simple Lie algebras we show by computational
means that the simple Lie algebras of exceptional type have the Panyushev
property.  We computationally prove two propositions on the dimension
of the abelianization of the centralizer of a nilpotent element in simple  Lie
algebras of exceptional type. Finally we obtain the closure ordering of the
orbits in the null cone of the spinor representation of the group
$\mathrm{Spin}_{13}(\C)$. All input and output of the relevant \GAP\ sessions
is given. 
\end{abstract}

\maketitle

\section{Introduction}

This paper has two purposes. Firstly, it serves to introduce and advertise
the capabilities of the computer algebra system {\GAP}4 \cite{GAP4} to perform
computations related to various aspects of Lie theory. The main objects
related to Lie theory that {\GAP} can deal with directly are Lie algebras and
related finite structures such as root systems and Weyl groups. 
But Lie algebras play 
an important role in the study of the structure and representations of
linear algebraic groups. So also the algorithms implemented in \GAP\
can also be used to perform computations regarding those objects.

The second purpose of the paper is to describe the results of three
computational projects that I have been involved in. The first of these
is the subject of Section \ref{sec:reach} and concerns reachable nilpotent
orbits in Lie algebras of exceptional type. Let $\g$ be a semisimple complex Lie
algebra and let $e\in \g$ be nilpotent. By $\g_e$ we denote the centralizer of
$e$ in $\g$. The element $e$ is said to be {\em reachable} if $e\in [\g_e,
\g_e]$. A nilpotent $e$ lies in a so-called $\ssl_2$-triple, which defines
a grading on $\g$. Panyushev \cite{panyushev4} proposed a characterization of
reachable nilpotent elements in terms of this grading; here we call this the
Panyushev property of $\g$. In \cite{panyushev4} this property was proved for
Lie algebras of type $A$. Yakimova \cite{yakimova2} showed that the Lie
algebras of type $B$, $C$, $D$ also have the Panyushev property. In Section
\ref{sec:reach} we show by calculations in \GAP\ that the simple Lie algebras
of exceptional type also have the Panyushev property.

The second project concerns the quotients $\g_e/[\g_e,\g_e]$ where again $e$ is
a nilpotent element in a simple complex Lie algebra $\g$. These play an
important role in \cite{pretop}. In Section \ref{sec:Qe} we show that a
statement proved in
\cite{pretop} for the simple Lie algebras of classical type also holds
for the exceptional types, albeit with a few explicitly listed exceptions.
The results of Sections \ref{sec:reach}, \ref{sec:Qe} have also appeared in the
arxiv preprint \cite{slapaper}, without giving the details of the
computations.

In Section \ref{sec:clos0} we look at the null cone of the spinor representation
of the group $\mathrm{Spin}_{13}(\C)$. The orbits of this group in the null
cone were first listed in \cite{gavin}. A' Campo and Popov
\cite[Example (f), p. 348]{dekem3} observed, also by computational means, that
these orbits coincide with the strata of the null cone (and they corrected
the dimensions given in \cite{gavin}). Here we show how an algorithm given
in \cite{gravinya} can be extended to this case to obtain the closure
ordering of these orbits. We give a simple implementation in \GAP\ and
obtain the closure diagram. Furthermore, we use \GAP\ to study the stabilizers
of the elements of the null cone.

In this paper we will not give a full introduction into working with Lie
algebras in \GAP\ but refer to the reference manual of \GAP\ which can be
found on its website, and to the manuals of the various packages that are
listed in the next section. The website of \GAP\ also has various introductory
materials of a more general nature. The topics that we discuss in this paper all
involve semisimple Lie algebras. For a general introduction to the theory
of these algebras we refer to the book by Humphreys, \cite{hum}.

We will give all input and output of the \GAP\ sessions. Most commands return
very quickly. If a command takes markedly longer then we display the runtime,
by using the \GAP\ function \verb+time+; this command displays the runtime
in milliseconds, so that a value of, for example, \verb+23345+ means 23.3
seconds. 

\noindent{\bf Acknowledgements.} I thank Alexander Elashvili for suggesting the
topic of Section \ref{sec:reach} and Alexander Premet for suggesting the
computations reported on in Section \ref{sec:Qe}. I thank the anonymous referee
for many comments which helped to improve the exposition of the paper.

\section{Preliminaries}

{\GAP}4 \cite{GAP4} is an open source computational algebra system.
Its mathematical functionality is contained in a ``core system''
(which consists of a small kernel written in C and a library of functions
written in the \GAP\ language) and a rather large number of packages which
can be loaded separately. The \GAP\ library has a number of functions
for constructing and working with Lie algebras and their representations.
For an overview we refer to the reference manual of \GAP. Furthermore
there are the following packages that deal with various aspects of Lie theory:
\begin{itemize}
\item {\sf CoReLG} \cite{corelg}, for working with real semisimple Lie algebras.
\item {\sf FPLSA} \cite{fplsa}, for dealing with finitely presented Lie
  algebras.
\item {\sf LieAlgDB} \cite{liealgdb}, which contains various databases of
  small dimensional Lie algebras.
\item {\sf LiePRing}, \cite{liepring} containing a database and algorithms
  for Lie $p$-rings.
\item {\sf LieRing} \cite{liering}, for computing with Lie rings.
\item {\sf NoCK} \cite{nock}, for the computation of Tolzano’s obstruction
  for compact Clifford-Klein forms.
\item {\sf QuaGroup} \cite{quagroup}, for computations with quantum groups.
\item {\sf SLA} \cite{sla}, for computations with various aspects of semisimple
  Lie algebras.
\item {\sf Sophus} \cite{sophus}, for computations in nilpotent Lie algebras.  
\end{itemize}

We also mention the package {\sf CHEVIE} for dealing with groups of Lie type
and related structures such as Weyl groups and Iwahori-Hecke algebras.
This package is built on \GAP 3, not \GAP 4. We refer to its website
\url{https://webusers.imj-prg.fr/~jean.michel/chevie/chevie.html} for
more information. 

The projects discussed in this paper mainly use the \GAP\ core system and
the package {\sf SLA}. In the next two subsections we briefly look at how
simple Lie algebras and their modules are constructed in \GAP\ and how
{\sf SLA} deals with nilpotent orbits in simple Lie algebras. 

\subsection{Simple Lie algebras in {\sf GAP}}

\GAP\ has a function \verb1SimpleLieAlgebra1 for creating the simple split
Lie algebras over fields of characteristic 0. (The semisimple Lie algebras
can be constructed by the function \verb1DirectSumOfAlgebras1.)
They are given by a
multiplication table with respect to a Chevalley basis (for the latter
concept we refer to \cite[Theorem 25.2]{hum}). For the base field we
usually take the rational numbers because often the computations
with these algebras are entirely rational, that is, require no solutions
to polynomial equations. The $i$-th basis element of such a Lie algebra is
written as \verb1v.i1. In the next example we construct the simple Lie
algebra of type $E_8$, a basis of it and two of its elements.

\begin{verbatim}
gap> L:= SimpleLieAlgebra( "E", 8, Rationals ); 
<Lie algebra of dimension 248 over Rationals>
gap> b:= Basis(L);; b[123];
v.123
gap> b[2]-3*b[5]+1/7*b[100];
v.2+(-3)*v.5+(1/7)*v.100
\end{verbatim}

Such simple Lie algebras come with a lot of data like a Chevalley basis
and a root system. Again we refer to the reference manual for more details. 

There also is a function for constructing the irreducible modules
of a semisimple Lie algebra. Such a module is given by a highest weight,
which is a nonnegative integral linear combination of the fundamental weights.
This linear combination is just given by its coefficient vector.
(The order of the fundamental weights is given by the Cartan matrix of
the root system of the Lie algebra.) The action of an element of the Lie
algebra on an element of its module is computed by the infix caret operator
\verb1^1. In the next example we construct the irreducible 3875-dimensional
module of the Lie algebra of type $E_8$. We see that the computation in
\GAP\ takes about 174 seconds. We also compute the action of an element of
the Lie algebra on an element of the module. 

\begin{verbatim}
gap> L:= SimpleLieAlgebra( "E", 8, Rationals );;
gap> V:= HighestWeightModule( L, [1,0,0,0,0,0,0,0] ); time;
<3875-dimensional left-module over <Lie algebra of dimension 248 
over Rationals>>
174425
gap> bL:= Basis(L);; bV:= Basis(V);;
gap> bL[1]^bV[263];
-1*y112*v0
\end{verbatim}

(For an explanation of the notation of the basis elements of these modules
we again refer to the reference manual.)

\subsection{Nilpotent orbits in {\sf GAP}}\label{sec:gapnilo}

Here we recall some definitions and facts on nilpotent orbits.
For more background information we refer to the the book by Collingwood and
McGovern (\cite{colmcgov}). Secondly we show how the package {\sf SLA}
deals with nilpotent orbits.

Let $\g$ be a semisimple Lie algebra over $\C$ (or over an algebraically
closed field of characteristic 0). Let $G$ denote the adjoint group of $\g$;
this is the identity component of the automorphism group of $\g$. An $e\in\g$
is said to be nilpotent if the adjoint map $\ad e : \g\to\g$ is nilpotent.
If $e\in G$ is nilpotent then the entire orbit $Ge$ consists of nilpotent
elements, and is therefore called a nilpotent orbit. 

By the Jacobson-Morozov theorem a nilpotent $e\in \g$ lies in an
$\ssl_2$-triple $(f,h,e)$ (where $[e,f]=h$, $[h,e]=2e$, $[h,f]=-2f$).
Let $\hh\subset\g$ be a Cartan subalgebra containing $h$. Let $\Phi$ be the root
system of $\g$ with respect to $\hh$. Then there is a basis of simple roots 
$\Delta = \{\alpha_1,\ldots,\alpha_\ell\}$ of $\Phi$, such that $\alpha_i(h) \in \{0,1,2\}$. 
The Dynkin diagram of $\Delta$,
where the node corresponding to $\alpha_i$ is labeled $\alpha_i(h)$, is called a
weighted Dynkin diagram. It uniquely determines the orbit $Ge$.

The nilpotent orbits of the simple Lie algebras have been classified,
see \cite{colmcgov}. In the {\sf SLA} package they can be constructed with the
command \verb1NilpotentOrbits1. The output is a list of objects that
carry some information such as the weighted Dynkin diagram of the orbit and
an $\ssl_2$-triple containing a representative. Here is an example for the
Lie algebra of type $E_7$, where we inspect the weighted Dynkin diagram and
the third element of an $\ssl_2$-triple of the 37-th orbit (that is, a
representative of the nilpotent orbit).

\begin{verbatim}
gap> L:= SimpleLieAlgebra("E",7,Rationals);;
gap> no:= NilpotentOrbits(L);;
gap> Length(no);
44
gap> WeightedDynkinDiagram( no[37] ); 
[ 2, 0, 0, 2, 0, 0, 2 ]
gap> SL2Triple( no[37] )[3];
v.8+v.11+v.13+v.15+v.22+v.23+v.24
\end{verbatim}

Now we briefly describe the concept of {\em induced} nilpotent orbit.

A subalgebra of $\g$ is said to be {\em parabolic} if it contains a Borel
subalgebra (i.e., a maximal solvable subalgebra). Let $\h$ be a fixed
Cartan subalgebra of $\g$. Let $\Phi$ denote the root system of $\g$ with
respect to $\h$, and let $\Delta$ be a fixed set of simple roots.
For a root $\alpha$ we denote the corresponding root space in $\g$ by
$\g_\alpha$. For a subset $\Pi\subset\Delta$ we define $\p_\Pi$ to be the
subalgebra generated by $\h$, $\g_{-\alpha}$ for $\alpha\in\Pi$ and
$\g_\alpha$ for all positive roots $\alpha$. Then $\p_\Pi$ is a parabolic
subalgebra. Furthermore, every parabolic subalgebra is $G$-conjugate to
a subalgebra of the form $\p_\Pi$.

Let $\p=\p_\Pi$ for a subset $\Pi\subset \Delta$. Let $\Psi \subset \Phi$
be the root subsystem that consists of the roots that are linear combinations
of the elements of $\Pi$. Then $\p = \myl\oplus\mf{n}$ where $\myl$ is the
subalgebra spanned by $\h$ and $\g_\alpha$ for $\alpha\in\Psi$. Secondly,
$\mf{n}$ is spanned by $\g_\alpha$ for positive $\alpha$ that do not lie in
$\Psi$. The decomposition $\p = \myl\oplus\mf{n}$ is called the Levi
decomposition of $\p$ and the subalgebra $\myl$ is called a (standard)
Levi subalgebra of $\g$. We observe that $\myl$ is a reductive Lie algebra.
In the sequel nilpotent orbits in Levi subalgebras appear. The definitions of
their properties are the obvious analogues of the definitions concerning
semisimple Lie algebras. 

Now let $\mf{p}\subset \g$ be a parabolic subalgebra, with Levi decomposition
$\mf{p} = \myl\oplus \mf{n}$. Let $L\subset G$ be the
connected subgroup of $G$ with Lie algebra $\myl$. Let $Le'$ be a nilpotent
orbit in $\myl$. Lusztig and Spaltenstein (\cite{lusp}) have shown that there
is a unique nilpotent orbit $Ge\subset \g$ such that
$Ge\cap (Le'\oplus \mf{n})$ is open and nonempty in  $Le'\oplus \mf{n}$.
The orbit $Ge$ is said to be {\em induced} from the orbit $Le'$. Nilpotent
orbits which are not induced are called {\em rigid}.

Let $n$ be a non-negative integer. The irreducible components of the locally
closed set 
$$A^n = \{ x \in \g \mid \dim Gx = n\}$$
are called {\em sheets} of $\g$ (see \cite{borho}, \cite{borhokraft}). A sheet
is $G$-stable and contains a {\em unique} nilpotent orbit. Sheets in general
are not disjoint, and different sheets may contain the same nilpotent orbit.
The sheets of $\g$ are indexed by $G$-classes of pairs $(\myl, Le')$, where
$\myl$ is a Levi subalgebra, and $Le'$ is a rigid nilpotent orbit in $\myl$,
see \cite{borho}. The nilpotent orbit that is contained in the corresponding
sheet is equal to the nilpotent orbit induced from $Le'$. 
The {\em rank} of the sheet corresponding to the pair $(\myl, Le')$ is defined
to be the dimension of the centre of $\myl$. 

In the {\sf SLA} package a sheet is represented by a {\em sheet diagram}.
We first explain how this is defined.
Consider a parabolic subalgebra $\p=\p_\Pi$ with corresponding Levi subalgebra
$\myl$. Let $Le'$ be a rigid nilpotent orbit in $\myl$, then the pair
$(\myl,Le')$ corresponds to a sheet. Now we label the Dynkin diagram of
$\Phi$ in the following way. Write $\Delta = \{\alpha_1,\ldots,\alpha_\ell\}$.
If $\alpha_i\not\in\Pi$ then node $i$ has label 2. The subdiagram consisting
of the nodes $i$ such that $\alpha_i\in \Pi$ is the Dynkin diagram of the
semisimple part of $\myl$. To these nodes we attach the labels of the
weighted Dynkin diagram of $Le'$. It is known that the weighted Dynkin diagram
of a rigid nilpotent orbit only has labels 0,1. So from a sheet diagram we
can identify $\myl$ and $Le'$ and hence the corresponding sheet.

The {\sf SLA} package has a function \verb1InducedNilpotentOrbits1 for
computing the induced nilpotent orbits of a simple Lie algebra. This function
returns a list of records that is in bijection with the sheets of the
Lie algebra. Each record has two components: \verb1norbit1 which is the
nilpotent orbit contained in the sheet, and \verb1sheetdiag1 which is the
list of labels of the sheet diagram of the sheet. Here is an example for
the simple Lie algebra of type $E_7$. 

\begin{verbatim}
gap> L:= SimpleLieAlgebra( "E", 7, Rationals );;
gap> ind:= InducedNilpotentOrbits( L );;
gap> Length( ind );
46
gap> ind[34];
rec( norbit := <nilpotent orbit in Lie algebra of type E7>, 
  sheetdiag := [ 2, 0, 0, 1, 0, 2, 2 ] )
gap> WeightedDynkinDiagram( ind[19].norbit );
[ 0, 0, 0, 2, 0, 0, 2 ]
gap> WeightedDynkinDiagram( ind[22].norbit );
[ 0, 0, 0, 2, 0, 0, 2 ]
gap> WeightedDynkinDiagram( ind[34].norbit );
[ 0, 0, 0, 2, 0, 0, 2 ]
\end{verbatim}

The numbering of the nodes of the Dynkin diagram of the Lie algebra
of type $E_7$ follows \cite[\S 11.4]{hum}. Hence the sheet diagram of the
34-th sheet is
$$2~~~~0~~~~\overset{\text{\normalsize 0}}{1}~~~~0~~~~2~~~~2$$
We obtain the Dynkin diagram of the corresponding Levi subalgebra $\myl$
by removing the nodes labeled 2; wee see that its semisimple part is of
type $D_4$. The weighted Dynkin diagram of the rigid nilpotent orbit
in $\myl$ has a 1 on the central node and zeros elsewhere. The rank of the
sheet is the dimension of the centre of $\myl$; this is the number of 2's
in the diagram, and we see that it is 3.

Furthermore we see that sheets 19 and 22 contain the same nilpotent orbit.
By inspection it can be verified that there are no other sheets that contain
this nilpotent orbit. Hence this is a nilpotent orbit lying in three sheets.

\section{Reachable elements}\label{sec:reach}

For $e\in \g$ we denote its centralizer
in $\g$ by $\g_e$. In \cite{panyushev4} an $e$ in $\g$ is defined to be
{\em reachable} if $e \in [\g_e,\g_e]$. Such an element has to be nilpotent.
It is obvious that $e$ is reachable if and only if all elements in its
orbit are reachable. Hence if $e$ is reachable then we also say that its
orbit $Ge$ is reachable.

In \cite{elashgre}, Elashvili and Gr\'elaud listed the reachable orbits in
simple complex Lie algebras $\g$ (in that paper reachable elements are
called {\em compact}, in analogy with \cite{blabry}).
For a given semisimple Lie algebra we can easily obtain this classification
in {\sf GAP}4, using the {\sf SLA} package. Here is an example for the simple
Lie algebra of type $E_6$.

\begin{verbatim}
gap> L:= SimpleLieAlgebra( "E", 6, Rationals );;
gap> nL:= NilpotentOrbits( L );;
gap> reach:= [ ];;
gap> for o in nL do
> e:= SL2Triple( o )[3]; ge:= LieCentralizer( L, Subalgebra(L,[e]) );
> if e in LieDerivedSubalgebra( ge ) then Add( reach, o ); fi;
> od;
gap> Length( reach );
6
gap> WeightedDynkinDiagram( reach[3] );
[ 0, 0, 0, 1, 0, 0 ]
\end{verbatim}

This simple procedure obtains six reachable nilpotent orbits. For each such
orbit we can look at its weighted Dynkin diagram to identify it in the known
lists of nilpotent orbits as in \cite[\S 8.4]{colmcgov}. The third element of
our list of reachable orbits corresponds to the orbit with label $3A_1$ in
the list in \cite[\S 8.4]{colmcgov}. 

Let $e\in\g$ be nilpotent, lying in the $\ssl_2$-triple $(f,h,e)$. The 
subalgebra spanned by $(f,h,e)$ acts on $\g$ (by restricting the adjoint
representation of $\g$). By the representation theory of $\ssl_2(\C)$ 
the eigenvalues of $\ad h$ are integers. Hence we get a grading
$$ \g = \bigoplus_{k\in\Z} \g(k)$$
where $\g(k) = \{ x \in \g\mid [h,x]=kx\}$. Now set $\g(k)_e = \g(k)\cap 
\g_e$, and let $\g(\geq 1)_e$ denote the subalgebra spanned by all
$\g(k)_e$, $k\geq 1$.

Panyushev (\cite{panyushev4}) showed that, for $\g$ of type $A_n$,
$e$ is reachable if and only if $\g(\geq 1)_e$ is generated as Lie algebra
by $\g(1)_e$. Here we call this the {\em Panyushev property} of $\g$.
In \cite{panyushev4} it is stated that this property also holds for the 
other classical types and the question is posed whether it holds for the
exceptional types. In \cite{yakimova2} a proof is given that the Panyushev
property holds in types $B_n$, $C_n$, $D_n$. Computations in {\sf GAP}
show that it also holds for the exceptional types. 

\begin{proposition}\label{prop:pan}
Let $\g$ be a simple Lie algebra of exceptional type. Then $\g$ has the
Panyushev property.  
\end{proposition}

\begin{proof}
One direction is easily seen to hold in general. Indeed,  
suppose that if $\g(\geq 1)_e$ is generated as Lie algebra by $\g(1)_e$.
Since $e\in \g(2)_e$ it immediately follows that $e$ is reachable.

The converse is shown by case by case computations in {\sf GAP}.
Here we show this for the Lie algebra of type $E_6$. We let {\tt reach} be
the list of reachable nilpotent orbits, as computed above.
\begin{verbatim}
gap> for o in reach do
> e:= SL2Triple( o )[3]; ge:= LieCentralizer( L, Subalgebra(L,[e]) );
> h:= SL2Triple( o )[2]; gr:= SL2Grading( L, h );
> gegeq1:= Intersection( ge, Subspace( L, Concatenation( gr[1] ) ) );
> ge1:= Intersection( ge, Subspace( L, gr[1][1] ) );
> Print( Subalgebra( L, Basis(ge1) ) = gegeq1, " " );
> od;
true true true true true true 
\end{verbatim}
The identifier {\tt gr} contains the grading corresponding to the 
$\ssl_2$-triple. This is a list consisting of three lists. The first of these
has bases of the subspaces $\g(1),\g(2),\ldots$. So $\g(\geq 1)_e$ is the
intersection of $\g_e$ and the subspace spanned by all elements in the union
of the lists in {\tt gr[1]}; this space is assigned to the identifier
\verb1gegeq11. Secondly, $\g(1)_e$ is the intersection of
$\g_e$ and the subspace spanned by the first element of {\tt gr[1]}; this
space is assigned to \verb1ge11. 
The penultimate line instructs {\sf GAP} to print {\tt true} if the
subalgebra generated by $\g(1)_e$ is equal to $\g(\geq 1)_e$. 
\end{proof}

Yakimova (\cite{yakimova2}) studied the stronger condition $\g_e = 
[\g_e,\g_e]$. In this paper we call elements $e$
satisfying this condition {\em strongly reachable}. She showed that
for $\g$ of classical type, $e$ is strongly reachable if and only if
the nilpotent orbit of $e$ is rigid. By an explicit example
this is shown to fail for $\g$ of exceptional type.
For the exceptional types we can show the following.

\begin{proposition}\label{prop:reachrig}
Let $\g$ be a simple Lie algebra of exceptional type. Let $e\in \g$ be
nilpotent. Then $e$ is strongly reachable if and only if $e$ is both
reachable and rigid.
\end{proposition}

\begin{proof}
If $e$ is strongly reachable then it is reachable, but also rigid by
\cite{yakimova2}, Proposition 11. As the {\sf SLA} package has a function for
determining the rigid nilpotent orbits, the converse can easily be shown
by direct computation.
But it also follows from Proposition \ref{prop:pan}. 
Indeed, if $e$ is rigid then $\g(0)_e$ is semisimple, so
$[\g(0)_e,\g(0)_e]=\g(0)_e$. Furthermore, $[\g(0)_e,\g(1)_e]= 
\g(1)_e$ by \cite[Lemma 8]{yakimova2} (where this is shown to hold for 
all nilpotent $e$).
By the Panyushev property this implies that $[\g_e,\g_e]= \g_e$.
\end{proof}

\begin{remark}
We can easily compute the rigid nilpotent orbits that are not strongly reachable.
Here is an example for the Lie algebra of type $E_8$.

\begin{verbatim}
gap> L:= SimpleLieAlgebra( "E", 8, Rationals );;
gap> rig:= RigidNilpotentOrbits( L );;
gap> exc:= [ ];;
gap> for o in rig do 
> e:= SL2Triple( o )[3]; ge:= LieCentralizer( L, Subalgebra(L,[e]) );
> if ge <> LieDerivedSubalgebra(ge) then Add( exc, o ); fi;
> od;
gap> Length( exc );
3
gap> WeightedDynkinDiagram( exc[1] );
[ 0, 0, 0, 0, 0, 1, 0, 1 ]
\end{verbatim}
We see that we have obtained three nilpotent orbits that are rigid but
not strongly reachable. Comparing the weighted Dynkin diagram of the first
of those orbits with the tables in \cite{colmcgov} we see that its Bala-Carter
label is $A_3+A_1$. Table \ref{tab:rignsr} contains the rigid but not strongly
reachable orbits in the Lie algebras of exceptional type;
it is used in the proof of \cite[Lemma 3.7]{preste}. For
an explanation of the notation used for the labels we refer to
\cite[\S 8.4]{colmcgov}.

\begin{table}[htb]
\begin{tabular}{|c|c|c|c|c|c|c|}
  \hline
  type & $E_7$ & $E_8$ & $E_8$ & $E_8$ & $F_4$ & $G_2$\\
  \hline
  label \phantom{{\huge B}} 
  & {\small $(A_3+A_1)'$} & {\small $A_3+A_1$} & {\small $D_5(a_1)+A_2$} &
  {\small $A_5+A_1$} & {\small $\widetilde{A}_2+A_1$} & {\small $A_1$} \\
\hline
{\small $(\dim \g_e, \dim [\g_e,\g_e])$} & (41,40) & (84,83) & (46,45) & (46,45) &
(16,15) & (6,5) \\
\hline
\end{tabular}
\caption{Rigid but not strongly reachable nilpotent orbits}\label{tab:rignsr}  
\end{table}
From the last line we see that in all cases $[\g_e,\g_e]$ is of codimension 1
in $\g_e$. Taking Proposition \ref{prop:reachrig} into account we see that
this implies that $\g_e = \langle e \rangle \oplus [\g_e,\g_e]$. In
\cite{preste} the $e$ with this property are called {\em almost reachable}.
\end{remark}  

\section{The quotients $\mf{c}_e$}\label{sec:Qe}

Let $\g$ be a simple Lie algebra, and $e$ a representative of a nilpotent orbit.
As before we denote its centralizer by $\g_e$. In this section we consider the
quotient $\mf{c}_e = \g_e/[\g_e,\g_e]$. These have been studied by Premet and
Topley \cite{pretop} in relation to finite $W$-algebras. In \cite{pretop}
it is shown that the statement of Proposition \ref{prop:sht} holds without 
exceptions for the classical Lie algebras. Proposition \ref{prop:sht}, as well as
the tables of \cite[Section 3]{slapaper}, are used in \cite{pretop} for showing
that for $\g$ of exceptional type, $U(\g,e)^{\mathrm{ab}}$ (the abelianization of
a finite $W$-algebra $U(\g,e)$) is isomorphic to a polynomial ring (with the
same six exceptions as Proposition \ref{prop:sht}).

\begin{proposition}\label{prop:sht}
Let $\g$ be a simple Lie algebra of exceptional type. Let $e\in \g$ be a
representative of an induced nilpotent orbit lying in a unique sheet. Then the
rank of that sheet is equal to $\dim \mf{c}_e$, except the
cases listed in Table \ref{tab:prop1}.
\end{proposition}

\begin{table}[htb]
\begin{tabular}{|l|l|l|c|c|}
\hline

$\g$ & label & weighted Dynkin diagram & rank & $\dim \mf{c}_e$ \\
\hline

$E_6$ & $A_3+A_1$ & $0~~~~1~~~~\overset{\text{\normalsize 1}}{0}~~~~1~~~~0$ 
& 1 & 2 \\
$E_7$ & $D_6(a_2)$ & $0~~~~1~~~~\overset{\text{\normalsize 1}}{0}~~~~1~~~~0~~~~2$
& 2 & 3\\
$E_8$ & $D_6(a_2)$ & $0~~~~1~~~~\overset{\text{\normalsize 1}}{0}~~~~0~~~~0~~~~1~~~~0$
& 1 & 3\\
$E_8$ & $E_6(a_3)+A_1$ & $1~~~~0~~~~\overset{\text{\normalsize 0}}{0}~~~~1~~~~0~~~~1~~~~0$
& 1 & 3\\
$E_8$ & $E_7(a_2)$ & $0~~~~1~~~~\overset{\text{\normalsize 1}}{0}~~~~1~~~~0~~~~2~~~~2$
& 3 & 4\\
$F_4$ & $C_3(a_1)$ &  1~~~0~~~1~~~0 & 1 & 3\\
\hline
\end{tabular}
\caption{Table of exceptions to Proposition \ref{prop:sht}.}\label{tab:prop1}
\end{table}

\begin{proof}
The proof is obtained by explicit computations in \GAP\ with the {\sf SLA}
package loaded.  We show the computation for the Lie algebra of type $E_8$.
First we compute the list of sheets (as explained in Section
\ref{sec:gapnilo}). For each sheet we compute $\dim \mf{c}_e$, where
$e$ is a representative of the unique nilpotent orbit in the sheet. These
dimensions are stored in the list \verb1dims1. Secondly, for each sheet we
compute the number of sheets having the same nilpotent orbit as the given
sheet. This number is stored in the list \verb1nr1. 
\begin{verbatim}  
gap> L:= SimpleLieAlgebra( "E", 8, Rationals );;
gap> shts:= InducedNilpotentOrbits( L );;
gap> nr:= [ ];; dims:= [ ];;
gap> for s in shts do            
> e:= SL2Triple( s.norbit )[3]; 
> ge:= LieCentralizer( L, Subalgebra( L, [e] ) );
> Add( dims, Dimension(ge)-Dimension(LieDerivedSubalgebra(ge)) );
> Add( nr, Length( Filtered( shts, t -> t.norbit = s.norbit ) ) );      
> od;
\end{verbatim}

For each sheet whose nilpotent orbit lies in no other sheet (that is, the
corresponding element of \verb1nr1 is 1) we compute its rank, which is
equal to the number of 2's in its sheet diagram (see Section \ref{sec:gapnilo}).
If the rank is not equal to $\dim \mf{c}_e$ then we store the sheet
in the list \verb1exc1. At the end this list contains the elements of Table
\ref{tab:prop1}. 

\begin{verbatim}
gap> exc:= [ ];;       
gap> for i in [1..Length(shts)] do
> if nr[i]=1 then                               
> rk:= Length( Filtered( shts[i].sheetdiag, x -> x = 2 ) );
> if rk <> dims[i] then Add( exc, shts[i] ); fi;
> fi; od;
gap> WeightedDynkinDiagram( exc[1].norbit );
[ 0, 1, 1, 0, 1, 0, 2, 2 ]
gap> Length( Filtered( exc[1].sheetdiag, x -> x = 2 ) );
3
gap> Position( shts, exc[1] );
8
gap> dims[8];
4
\end{verbatim}
So we have obtained the data of the penultimate line of Table \ref{tab:prop1}.
\end{proof}

\begin{proposition}
Let $\g$ be a simple Lie algebra of exceptional type, and let $e\in \g$ be a
nilpotent orbit that lies in more than one sheet. Then the maximal rank of
such a sheet is strictly smaller than $\dim \mf{c}_e$.
\end{proposition}

\begin{proof}
Also this proposition is proved by direct computation. Again we show the
computation for the simple Lie algebra of type $E_8$. We assume that the first
part of the computation explained in the proof of the previous proposition
has been done. In this case, for each sheet such that the corresponding number
in \verb1nr1 is greater than one, we first determine all sheets that have the
same nilpotent orbit. This list is assigned to the identifier \verb1sh1.
Then we compute the rank of all sheets in \verb1sh1. If the maximum of 
those ranks is not strictly smaller than $\dim \mf{c}_e$ (where $e$ is a
representative of the nilpotent orbit in the considered sheet) then
we print a \verb1?1; otherwise we print a \verb1!1. Since we only obtain
\verb1!1, the proposition is proved in this case.
\begin{verbatim}
gap> for i in [1..Length(shts)] do
> if nr[i] > 1 then                                             
> sh:= Filtered( shts, t -> t.norbit = shts[i].norbit );
> rks:= List( sh, r -> Length( Filtered( r.sheetdiag, x -> x=2 ) ) );
> if Maximum( rks ) >= dims[i] then Print("?"); else Print("!"); fi;
> fi; od;
!!!!!!!!!!!!!!!!!!!!!!!!!!!!!!!!!!!!!!!
\end{verbatim}
\end{proof}

\begin{remark}
Let $e\in \g$ be nilpotent lying in the $\ssl_2$-triple $(f,h,e)$. The 
Jacobi identity implies that the adjoint map $\ad h : \g\to \g$ stabilizes
$\g_e$ and $[\g_e,\g_e]$. Hence it induces a map $\ad h : \mf{c}_e
\to \mf{c}_e$. The representation theory of $\ssl_2(\C)$ implies that
$\ad h$ acts with non-negative integral eigenvalues on $\mf{c}_e$.
The paper \cite{slapaper} contains tables listing those eigenvalues for
the nilpotent orbits of exceptional simple Lie algebras.
\end{remark}  

\section{Closures of nilpotent orbits of $\mathrm{Spin}_{13}$}\label{sec:clos0}

Let $G$ be a reductive complex algebraic group and let $V$ be a
finite-dimensional rational $G$-module. Then the invariant ring
$\C[V]^G$ is finitely generated by homogeneous elements. The null cone
$N_G(V)$ is defined to be the zero locus of the homogeneous invariants
of positive degree. The null cone is stable under the action of $G$ but in
general consists of an infinite number of orbits. Hesselink \cite{hesselink2}
constructed a stratification of the null cone, by which it is possible to
study its geometric properties. In \cite[\S 5.5, 5.6]{povin} Popov and
Vinberg gave a version
of this theory in characteristic 0 that works with certain elements,
called characteristics, in the Lie algebra of $G$. Popov \cite{popov}
developed an algorithm to compute these characteristics. 

The $G$-module $V$ is
said to be {\em visible} (or {\em observable}) if the null cone has a finite
number of orbits. Kac \cite{kacnilp} classified the visible representations of
reductive algebraic groups. It turns out that irreducible visible
representations of connected simple groups either
arise as so-called $\theta$-groups or as the spinor modules of
$\mathrm{Spin}_{11}(\C)$ and $\mathrm{Spin}_{13}(\C)$. For an algorithm for
determining the closures of the nilpotent orbits of a $\theta$-group we
refer to \cite{gravinya}. The orbits of the spinor module
$\mathrm{Spin}_{11}(\C)$ have been determined by Igusa \cite{igusa}.
It is likely that the closures of the nilpotent orbits can be determined
in the same way as is done below. 

Kac and Vinberg \cite{gavin} classified the orbits of the group
$\mathrm{Spin}_{13}(\C)$ on its 64-dimensional spinor module. It turns out
that the null cone has 13 orbits (excluding 0). A' Campo and Popov
\cite[Example (f), p. 348]{dekem3}, using their implementation of Popov's
algorithm \cite[Appendix C]{dekem3} for computing the characteristics
of the strata, observed that there are also 13 strata in the null cone. This
implies that the strata are orbits. Moreover, their computations gave the
dimensions of the orbits in the null cone, which were not all 
correctly given in \cite{gavin}.

The package {\sf SLA} also has an implementation of Popov's algorithm.
So we can recover these observations by a computation using that
package. In this section we give an algorithm, which is similar to an
algorithm given in \cite{gravinya}, to determine when the (Zariski-) closure
of a stratum contains a given other stratum. This algorithm works under
some hypotheses that are shown to be satisfied by the spinor module of
$\mathrm{Spin}_{13}(\C)$. We discuss a simple implementation of this
algorithm in {\sf GAP} and we obtain the Hasse diagram of the closures of
the orbits in the null cone of the spinor module of $\mathrm{Spin}_{13}(\C)$.

\subsection{Preliminaries on the strata of the nullcone}

Everything we will say here works for reductive groups, but for simplicity we
consider a simple algebraic group $G$ over $\C$. We let $\g$ be its Lie
algebra and $(~,~) : \g\times \g \to \C$ the Killing form (so
$(x,y) = \Tr((\ad x)(\ad y))$). We say that a semisimple element $h\in \g$
is {\em rational} if the eigenvalues of $\ad h$ lie in $\Q$. This is equivalent
to saying that the eigenvalues of $h$ on any $\g$-module are rational.

Let $\h\subset \g$ be a Cartan subalgebra. Then by $\h_\Q$ we denote the
set of its rational elements, which is a vector space over $\Q$ of dimension
$\dim_\C \h$. We define the norm of $h\in \h_\Q$ by $||h|| = \sqrt{(h,h)}$.

Now we let $V$ be a rational $G$-module and consider the null cone $N_G(V)$.
By the Hilbert-Mumford criterion a $v\in V$ lies in $N_G(V)$ if and
only if there is a cocharacter $\chi : \C^*\to G$ such that
$\lim_{t\to 0} \chi(t)\cdot v = 0$ (see \cite[Section III.2]{kraft}).
Setting $h = d \chi(1)$ we have that $h$ is
a rational semisimple element and writing $v$ as a sum of $h$-eigenvectors
we get that the corresponding eigenvalues are all positive.

For a rational semisimple $h\in \g$ and $\tau\in \Q$ we let $V_\tau$ be
the $\tau$-eigenspace of $h$. Furthermore, we set
$$V_{\geq 2} (h) = \bigoplus_{\tau\geq 2} V_\tau(h).$$
Let $v\in V$. Then a {\em characteristic} of $v$ is a shortest rational
semisimple element $h\in \g$ such that $v\in V_{\geq 2}(h)$.
We have the following facts concerning characteristics
(see \cite[\S 5.5, 5.6]{povin}, \cite[\S 7.4.1, 7.4.2]{gra16}):
\begin{enumerate}
\item $v$ has a characteristic if and only if $v\in N_G(V)$.
\item If $h\in \g$ is a characteristic of $v\in V$ and $g\in G$ then
  $\Ad(g)(h)$ is a characteristic of $gv$.
\item Let $\h$ be a fixed Cartan subalgebra of $\g$. Then there are a finite
  number of characteristics $h_1,\ldots,h_s$ in $\h$, up to the action of $G$.
\item For $1\leq i\leq s$ let $S(h_i)$ be the set of all $v\in N_G(V)$
  such that $v$ has a characteristic that is $G$-conjugate to $h_i$.
  The set $S(h_i)$ is called the stratum corresponding to $h_i$.
\item The stratification of $N_G(V)$ is $N_G(V) = S(h_1)\cup
  \cdots \cup S(h_s)$ (disjoint union).  
\end{enumerate}  

Popov \cite{popov} (see also \cite[\S 7.4.3]{gra16})
devised an algorithm to compute the characteristics 
$h_1,\ldots,h_s$ in $\h$. The algorithm also computes the dimensions of
the corresponding strata.

\subsection{Closures of the strata}\label{sec:clos}

The topological notions (closed sets, open sets, closure,...) that we use here
are relative to the Zariski topology.

Let $\h$ be a fixed Cartan subalgebra of $\g$. For a rational $h\in \h$
we let $Z(h) = \{ g\in G \mid \Ad(g)(h) = h\}$; then
$\z(h) = \{ x\in \g \mid [x,h]=0\}$ is the Lie algebra of $Z(h)$.
Both $Z(h)$ and $\z(h)$ stabilize the spaces $V_\tau(h)$ for $\tau\in \Q$. 

Let $h_1,\ldots,h_s\in\h$ be the characteristics of the strata of the 
nullcone of $V$.

Here we assume two things:
\begin{enumerate}
\item Each $V_2(h_i)$ has an open $Z(h_i)$-orbit.
\item The strata coincide with the $G$-orbits in the nullcone.
\end{enumerate}

\begin{remark}\label{rem:vinopen}
Let $h$ be one of the characteristics. A $v\in V_2(h)$ lies in the
open $Z(h)$-orbit if and only if $\z(h)\cdot v = V_2(h)$.  
\end{remark}  

Under these hypotheses we can generalize a few results from \cite{gravinya}.

\begin{lemma}\label{lem:char1}
Let $h$ be one of the characteristics. Then the open
$Z(h)$-orbit in $V_2(h)$ is equal to $V_2(h)\cap S(h)$. Moreover, $h$
is a characteristic of every element in $V_2(h)\cap S(h)$.
\end{lemma}

\begin{proof}
Let $u$ be an element of the open $Z(h)$-orbit in $V_2(h)$.
From  Theorem 5.4 in \cite{povin} it follows that the set of elements of
$V_2(h)$ with characteristic $h$ is open and nonempty. As nonempty
open sets intersect, there is a $g\in Z(h)$ such that
$g\cdot u$ has characteristic $h$. But then the characteristic of
$u=g^{-1}\cdot (gu)$ is $\Ad(g^{-1})(h)=h$. It follows that $h$ is a
characteristic of $u$, and in particular that $u \in S(h)$.

For $\tau\in \Q$ and $w\in V_2(h)$ set
$${\g}_w = \{ x\in \g \mid x\cdot w = 0\}, \text{ and } {\g}_{\tau,w} =
\{ x\in {\g}_{w} \mid [h,x]=\tau x\}.$$
Let $v\in V_2(h)\cap S(h)$. 
Since $v$ lies in the closure of $Z(h)u$, we have that $\dim \g_{\tau,v}\geq
\dim \g_{\tau, u}$, for all $\tau\in \Q$. Because $u,v\in S(h)$ and our
assumption that the strata are $G$-orbits, $v$ and $u$ lie in the same
$G$-orbit. Hence $\dim \g_{v} = \dim \g_{u}$. But $\g_v$ is the 
direct sum of the various $\g_{\tau,v}$, and similarly for $\g_{u}$. It 
follows that $\dim \g_{\tau,v} = \dim \g_{\tau, u}$ for all $\tau$. But
$${\g}_{0,v} = \{ x\in \z(h) \mid x\cdot v = 0\},$$
and similarly for $\g_{0,u}$. This implies that $\dim \z(h)v = \dim \z(h)u$.
So also the orbit $Z(h)v$ is open in $V_2(h)$ by Remark \ref{rem:vinopen}.
In particular, $v$ lies in the open $Z(h)$-orbit in $V_2(h)$.
\end{proof}

\begin{lemma}
Let $W$ denote the Weyl group of the root system of $\g$.  
Let $h,h'$ be two of the characteristics. Then $S(h')$ is 
contained in the closure of $S(h)$ if and only if there is a 
$w\in W$ such that $U=V_2(h')\cap V_{\geq 2}(wh)$ contains a point of 
$S(h')$. Furthermore, the intersection of $U$ and $S(h')$
is open in $U$.
\end{lemma}

\begin{proof}
Here we use the fact that $\overline{S(h)} = GV_{\geq 2}(h)$
(\cite{povin}, Theorem 5.6). This immediately implies the ``only if'' part.

Let $P(h)$ denote the parabolic subgroup with Lie algebra $\oplus_{\tau \geq 0}
\g_{\tau}(h)$. Using the Bruhat decomposition we then have
$$\overline{S(h)} = \bigcup_{w\in W} P(h') wP(h) (V_{\geq 2}(h))
= \bigcup_{w\in W} P(h')w(V_{\geq 2}(h)).$$
Suppose that $S(h') \subset \overline{S(h)}$. Let $v'\in
V_2(h')\cap S(h')$.
Then it follows that
there are $p\in P(h')$, $w\in W$, $v\in V_{\geq 2}(h)$ with $v'=pw\cdot v$,
or $p^{-1}\cdot v'= w\cdot v$. 

We have that $P(h') = Z(h)\ltimes N$, where $N$ is the unipotent subgroup of
$G$ with Lie algebra $\oplus_{\tau >0} \g_\tau(h)$. So $p^{-1} = zn$ with
$z\in Z(h)$, $n\in N$. As $v'\in V_2(h')$, we see that $nv' = v'+v''$ with
$v''\in V_{>2}(h')$. So $p^{-1}\cdot v' = zv'+zv''$ with $zv'\in V_2(h')$, 
$zv''\in V_{>2}(h')$. In particular, $p^{-1}\cdot v' \in V_{\geq 2}(h')$.
But $w\cdot v \in V_{\geq 2}(wh)$. So   $p^{-1}\cdot v' \in V_{\geq 2}(h')\cap
V_{\geq 2}(wh)$. Denote the latter space by $\widetilde{U}$.

Since $h'$ and $wh$ commute, $\widetilde{U}$ is stable under $h'$. So 
$\widetilde{U}$ is the direct sum of $h'$-eigenspaces. Hence $zv'\in 
\widetilde{U}$. So, in fact, $zv' \in U$, and obviously, $zv'\in S(h')$.

The last statement follows from \cite{povin}, Theorem 5.4.
\end{proof}

These lemmas underpin a direct method for checking whether $S(h') \subset
\overline{S(h)}$: 
\begin{enumerate}
\item For all $w\in W$ compute the space $U_w = V_2(h')\cap V_{\geq 2}(wh)$.
\item Take a random point $u\in U_w$. If $\dim \z(h')\cdot u = \dim
V_2(h')$, then conclude that $S(h')\subset \overline{S(h)}$.
\end{enumerate}

If in Step 2, the equality does not hold, then it is very likely that $U_w$
contains no point of $S(h')$. However, we still need to prove it. One method
for that is described in \cite[Section 5]{gravinya}, based on 
computing the generic rank of a matrix with polynomial entries. It also works
here. However, a different approach is also possible: Compute the weights 
$\mu_1,\ldots,\mu_r\in \h^*$
of the weight spaces whose sum is $U_w$. By using the form $(~,~)$ 
we obtain an isomorphism $\nu : \h\to \h^*$ by $\nu(x)(y) = (x,y)$.
We consider the Euclidean space $\h_\R = \R\otimes \h_\Q$ with inner product
$(~,~)$. Let $C$ be the convex hull in $\h_\R$ of the points
$\hat h_i=\nu^{-1}(\mu_i)$. Note that all $\hat h_i$ lie in the affine space
$H_2$ consisting of all $x\in \h_\R$ with $(h',x)=2$. So also $C\subset H_2$.
Let $\tau\in \Q$ be
such that $(h',\tau h')=2$, then also $\tau h'\in H_2$. Now if $\tau h'$
does not lie in $C$ then $U_w$ has no point of $S(h')$.
This follows from the following fact: let $u\in U_w$, and 
let $C'$ be the convex hull of $\nu^{-1}(\mu)$, where $\mu$ ranges over
the weights involved in an expression of $u$ as sum of weight vectors, and let
$\tilde h$ be the point on $C'$ closest to 0, and let $\hat h$ be such that
$(\tilde h,\hat h)=2$, then $\hat h$ is a characteristic of $u$, or 
$\h$ does not contain a characteristic of $u$ (see \cite[Section 5.5]{povin}
or \cite[Lemma 7.4.16]{gra16}).

\subsection{Implementation for $\mathrm{Spin}_{13}$}

The Lie algebra of $G=\mathrm{Spin}_{13}(\C)$ is the simple Lie algebra of
type $B_6$. We can construct this Lie algebra in {\sf GAP}. The nodes
of the Dynkin
diagram of the root system of this Lie algebra are numbered in the
usual way (see, for example,  \cite[\S 11.4]{hum}). Denoting the corresponding
fundamental weights by $\lambda_1,\ldots,\lambda_6$ we have that the
highest weight of the spinor module is $\lambda_6$. The {\sf SLA} package
contains the function \verb1CharacteristicsOfStrata1 which implements
Popov's algorithm. On input a semisimple Lie algebra and a dominant weight
it returns a list of two lists: the first is the list of characteristics,
the second is the list of dimensions of the corresponding strata.
In the next example we compute the characteristics of the strata of the spinor
module of $G$ (which takes about 87 seconds). 
With \verb1SortParallel1 we sort the list of dimensions, and
apply the same permutation to the list of characteristics.
We display the list of dimensions and the first characteristic,
which is an element of $\g$. Comparing with \cite[Example (f), p. 348]{dekem3}
we see that we get the same dimensions as A'Campo and Popov.

\begin{verbatim}
gap> L:= SimpleLieAlgebra("B",6,Rationals);;
gap> st:= CharacteristicsOfStrata( L, [0,0,0,0,0,1] );; time;
86818
gap> chars:= st[1];; dims:= st[2];;
gap> SortParallel( dims, chars );
gap> dims;
[ 22, 32, 35, 42, 43, 43, 46, 50, 50, 53, 56, 58, 62 ]
gap> chars[1];
(2/3)*v.73+(4/3)*v.74+(2)*v.75+(8/3)*v.76+(10/3)*v.77+(2)*v.78
\end{verbatim}

Above we already argued that the strata are $G$-orbits. In order to be
able to apply the algorithm of the previous section we need to show that
for each characteristic $h$ the space $V_2(h)$ has an open $Z(h)$-orbit.
For this we first construct the spinor module $V$ (this is done with the
{\sf GAP} function {\tt HighestWeightModule}). If {\tt x}, {\tt v} are
elements of the Lie algebra {\tt L} and the module {\tt V} respectively, then
\verb1x^v1 is the result of acting with {\tt x} on \verb1v1. Since
the basis elements of the module that is output by \verb1HighestWeightModule1
are weight vectors relative to the Cartan subalgebra of \verb1L1 that contains
the characteristics, the following function can be used to find a basis of
$V_2(h)$:

\begin{verbatim}
V2:= function( V, h )
  return Filtered( Basis(V), v -> h^v = 2*v );
end;
\end{verbatim}

Let $h$ be a characteristic, say the fifth one. We show that $V_2(h)$ has an
open $Z(h)$-orbit:

\begin{verbatim}
gap> V:= HighestWeightModule( L, [0,0,0,0,0,1] );;
gap> h:= chars[5];;
gap> v2:= V2( V, h );;
gap> v:= Sum( v2, x -> Random([-100..100])*x );;
gap> zh:= LieCentralizer( L, Subalgebra( L, [h] ) );;
gap> zhv:= Subspace( V, List( Basis(zh), x -> x^v ) );;
gap> Dimension( zhv ) = Length(v2);
true
\end{verbatim}

Here we take a random point $v$ of $V_2(h)$. We let \verb1zh1, \verb1zhv1
be the centralizer $\z(h)$ and the space $\z(h)\cdot v$ respectively.
The last line shows that $\dim \z(h)\cdot v=
\dim V_2(h)$. This implies that the orbit of $v$ is open in $V_2(h)$
(Remark \ref{rem:vinopen}). We have executed this procedure for all
characteristics, and hence both hypotheses of the previous section are
satisfied. 

Now in order to execute the procedure of the previous section we need
functions for computing $V_{\geq 2}(h)$ and $wh$ for $w$ in the Weyl group
$W$. The function for the former is straightforward:

\begin{verbatim}
Vgeq2:= function( V, h )

        local m,i;

        m:= MatrixOfAction( Basis(V), h );
        i:= Filtered( [1..Length(m)], i -> m[i][i] >= 2 );
        return Basis( V ){i};
end;
\end{verbatim}

That is, we take the matrix of \verb1h1 (which is diagonal) and return the
list of basis vectors that correspond to an eigenvalue which is at least 2.

In order to compute $wh$ we consider a Chevalley basis of $L$,
\cite[Theorem 25.2]{hum}. Such a basis consists of elements $x_\alpha$ for
$\alpha$ in the root system, and $h_1,\ldots,h_\ell$ that lie in the
Cartan subalgebra. We refer to the cited theorem for the multiplication
table with respect to this basis. For a root $\alpha$ we set $h_\alpha =
[x_\alpha,x_{-\alpha}]$. Then we have $wh_\alpha = h_{w\alpha}$. Furthermore, if
$\alpha_1,\ldots,\alpha_\ell$ are the simple roots then $h_{\alpha_i} = h_i$.

A simple Lie algebra in {\sf GAP}, constructed with the function
\verb+SimpleLieAlgebra+, has a stored Chevalley basis. This is
a list consisting of three lists. In the first list we have the $x_\alpha$ for
$\alpha$ a positive root. In the second list we have the $x_\alpha$ for
$\alpha$ a negative root. The third list has the elements $h_1,\ldots,h_\ell$.
The ordering that is used on the positive roots is height compatible
(cf. \cite[\S 10.1]{hum}). This means that the $x_{\alpha_i}$ for
$1\leq i\leq \ell$ come first. Denote the positive roots, as ordered by 
{\sf GAP}, by $\alpha_1,\ldots,\alpha_n$. For $n+1\leq i\leq 2n$ set $\alpha_i =
-\alpha_{i-n}$. The {\sf SLA} package has a function,
\verb1WeylGroupAsPermGroup1, that gives the Weyl group as a permutation
group on $1,\ldots ,2n$. If $w$ is an element of this group then the
corresponding element of the Weyl group acts as $\alpha_i\mapsto \alpha_{i^w}$.
These considerations yield the following function for computing $wh$,
where $w$ is given as a permutation and $h$ lies in the given Cartan
subalgebra. Here the first two input parameters are the following: \verb1BH1 is
the basis of the Cartan subalgebra with basis vectors $h_1,\ldots,h_l$;
\verb1hs1 is the list $h_{\alpha_i}$ for $1\leq i\leq 2n$. 

\begin{verbatim}
wh:= function( BH, hs, w, h )

        local cf, i;

        cf:= Coefficients( BH, h );
        i:= List( [1..Length(cf)], j -> j^w );
        return cf*hs{i};
end;
\end{verbatim}

With this preparation we can give the implementation of the algorithm
described in Section \ref{sec:clos}. Here we give the simplified probabilistic
version, where we do not prove the non-inclusions. (The complete version
is longer as it includes an implemtation of a function to check membership
of a convex hull. It has been used to prove the correctness of the diagram
in Figure \ref{fig:spin13}, and is available from the author upon request.)
We start by defining a number of global variables
that will be accessed by the function. Most of these have been explained above.
The list \verb1eW1 contains all elements of the Weyl group. The function
\verb1inc1 is a straightforward implementation of the algorithm
given in Section \ref{sec:clos}.

\begin{verbatim}
L:= SimpleLieAlgebra("B",6,Rationals);
st:= CharacteristicsOfStrata( L, [0,0,0,0,0,1] );
chars:= st[1];; dims:= st[2];;
SortParallel( dims, chars );

V:= HighestWeightModule(L,[0,0,0,0,0,1]);
R:= RootSystem(L);
ch:= ChevalleyBasis(L);
hs:= List( [1..36], i -> ch[1][i]*ch[2][i] );
hs:= Concatenation( hs, -hs );
h:= ch[3];
BH:= Basis( CartanSubalgebra(L), h );
eW:= Elements( WeylGroupAsPermGroup(R) );

inc:= function( h1, h2 )

        local v2, zh1, w, vgeq2, U, u;

        v2:= Subspace( V, V2( V, h1 ) );
        zh1:= BasisVectors( Basis( LieCentralizer( L, Subalgebra( L, [h1]))));
        for w in eW do
            vgeq2:= Subspace( V, Vgeq2( V, wh( BH, hs, w, h2 ) ) );
            U:= Intersection( v2, vgeq2 );
            if Dimension(U) > 0 then
               u:= Sum( Basis(U), x -> Random([-30..30])*x );
               if Subspace( V, List( zh1, x -> x^u)) = v2 then
                  return true;
               fi;
            fi;
        od;

        return false;

end;
\end{verbatim}

We now give a short example of the useage of this function.

\begin{verbatim}
gap> inc( chars[6], chars[9] ); time;
true
228
gap> inc( chars[6], chars[7] ); time;
false
3923358
\end{verbatim}

Here we see that the orbit with the sixth characteristic is contained in
the closure of the orbit with the ninth characteristic, but not in the
closure of the orbit with the seventh characteristic. The first computation
takes 0.2 seconds whereas the second computation takes 3923.3 seconds.
This is explained by the fact that for the second computation the entire
Weyl group is transversed, which has 46080 elements, whereas the
first computation is decided after considering just one element of the
Weyl group.

\subsection{Closure diagram and stabilizers}

By applying the implementation of the previous section we arrive at the
Hasse diagram in Figure \ref{fig:spin13} that displays the closure relation
of the orbits in the null cone.

\unitlength=1cm
\begin{figure}[htb]

\includegraphics{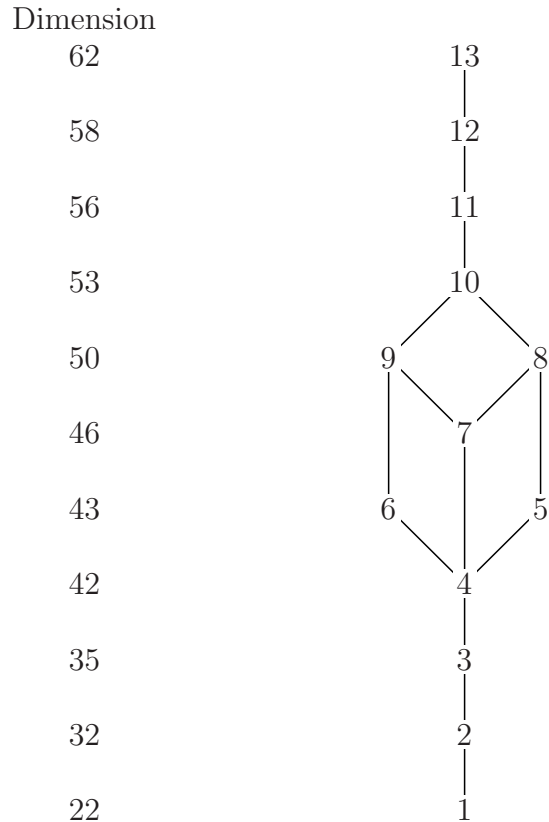}

\caption{Hasse diagram of the closures of the orbits of $\mathrm{Spin}_{13}$
  in the null cone}
\label{fig:spin13}
\end{figure}

Using Lemma \ref{lem:char1} it is straightforward to find representatives
of the orbits in the null cone. We illustrate this by an example:

\begin{verbatim}
gap> h:=chars[5];;
gap> v2:=V2( V, h );;       
gap> zh:= Basis( LieCentralizer( L, Subalgebra(L,[h])) );;      
gap> Length(v2);
32
gap> v:= v2[1]+v2[32];;
gap> Dimension( Subspace( V, List( zh, x -> x^v ) ) );
32
\end{verbatim}

This computation shows that the constructed element \verb1v1 is a
representative of the orbit corresponding to the fifth characteristic.
(We have found it by systematically trying sums of elements of \verb1v21;
here we do not go into that.)

Given an element $v\in V$ we can consider its stabilizer in $\g$:
$${\g}_{v} = \{ x\in \g \mid x\cdot v = 0\},$$
which is the Lie algebra of the stabilizer in $G$.
The {\sf SLA} package does not contain a function for computing this stabilizer,
but it is easily written:

\begin{verbatim}
stab:= function( v ) # v in V, we return its stabilizer in L

        local m, sol;

        m:= List( Basis(L), x -> Coefficients( Basis(V), x^v ) );
        sol:= NullspaceMat( m );
        return List( sol, x -> x*Basis(L) );
end;
\end{verbatim}

We then can use {\sf GAP} functionality to study the structure of the
stabilizer. We use the function \verb1LeviMalcevDecomposition1 which for a
Lie algebra \verb1K1 returns a list of two subalgebras. The first of these is
semisimple, the second is solvable and \verb1K1 is their semidirect sum.
In our example this goes as follows.

\begin{verbatim}
gap> K:= Subalgebra( L, stab(v) );;
gap> ld:=LeviMalcevDecomposition(K);;
gap> SemiSimpleType(ld[1]);  
"A4"
gap> Dimension(ld[2]);
11
\end{verbatim}

We see that the stabilizer is the semidirect product of a simple Lie algebra
of type $A_4$ and an 11-dimensional solvable ideal. By inspecting the basis
elements of this ideal it is easily seen that it is spanned by root vectors
corresponding to positive roots. Hence the ideal is unipotent. We indicate
this by saying that the stabilizer is of type $A_4\ltimes U_{11}$.
By doing similar computations for all 13 characteristics we arrive at Table
\ref{tab:stab}.

\begin{table}[htb]
  \begin{tabular}{|r|c|r|}
    \hline
    nr & dim & type of stabilizer\\
    \hline
    1 & 22 & $A_5\ltimes U_{21}$\\
    2 & 32 & $A_2+G_2\ltimes U_{24}$\\
    3 & 35 & $A_1+B_3\ltimes U_{19}$\\
    4 & 42 & $B_2+T_1\ltimes U_{25}$\\
    5 & 43 & $A_4\ltimes U_{11}$\\
    6 & 43 & $C_3\ltimes U_{14}$\\
    7 & 46 & $B_2\ltimes U_{22}$\\
    8 & 50 & $A_1+A_2\ltimes U_{17}$\\
    9 & 50 & $A_2\ltimes U_{20}$\\
    10 & 53 & $A_1+A_1\ltimes U_{19}$\\
    11 & 56 & $A_1+A_1\ltimes U_{16}$\\
    12 & 58 & $B_2\ltimes U_{10}$\\
    13 & 62 & $A_1\ltimes U_{13}$\\
    \hline
  \end{tabular}
  \caption{Stabilizers of the orbits in the null cone of the spinor
    representation of $\mathrm{Spin}_{13}$}\label{tab:stab}
\end{table}

We see that the sum of the dimension in the second column and the dimension
of the stabilizer is always $78=\dim\g$ (which should be the case as
$\dim \g_v + \dim \g\cdot v = \dim \g =78$).


\begin{thebibliography}{CdGSGT19}

\bibitem[BB92]{blabry}
Philippe Blanc and Jean-Luc Brylinski, \emph{Cyclic homology and the {S}elberg
  principle}, J. Funct. Anal. \textbf{109} (1992), no.~2, 289--330.

\bibitem[BJS{\etalchar{+}}19]{nock}
M.~Boche{\a'n}ski, P.~Jastrzkebski, A.~Szczepkowska, A.~Tralle, and A.~Woike,
  \emph{{NoCK}, nock-package for computing obstruction for compact
  {C}lifford-{K}lein forms., {V}ersion 1.4}, \href
  {https://pjastr.github.io/NoCK}
  {\texttt{https://pjastr.github.io/}\discretionary {}{}{}\texttt{NoCK}}, Oct
  2019, Refereed GAP package.

\bibitem[BK79]{borhokraft}
Walter Borho and Hanspeter Kraft, \emph{\"{U}ber {B}ahnen und deren
  {D}eformationen bei linearen {A}ktionen reduktiver {G}ruppen}, Comment. Math.
  Helv. \textbf{54} (1979), no.~1, 61--104.

\bibitem[Bor82]{borho}
Walter Borho, \emph{\"{U}ber {S}chichten halbeinfacher {L}ie-{A}lgebren},
  Invent. Math. \textbf{65} (1981/82), no.~2, 283--317.

\bibitem[CdGGT19]{liering}
S.~Cical{\a`o}, W.~A. de~Graaf, and T.~GAP~Team, \emph{{LieRing}, computing
  with finitely presented {L}ie rings, {V}ersion 2.4.1}, \href
  {https://gap-packages.github.io/liering/}
  {\texttt{https://gap-packages.github.io/}\discretionary
  {}{}{}\texttt{liering/}}, Feb 2019, Refereed GAP package.

\bibitem[CdGSGT19]{liealgdb}
S.~Cical{\a`o}, W.~A. de~Graaf, C.~Schneider, and T.~GAP~Team,
  \emph{{LieAlgDB}, a database of {L}ie algebras, {V}ersion 2.2.1}, \href
  {https://gap-packages.github.io/liealgdb/}
  {\texttt{https://gap-packages.github.io/}\discretionary
  {}{}{}\texttt{liealgdb/}}, Oct 2019, Refereed GAP package.

\bibitem[CM93]{colmcgov}
David~H. Collingwood and William~M. McGovern, \emph{Nilpotent orbits in
  semisimple {L}ie algebras}, Van Nostrand Reinhold Mathematics Series, Van
  Nostrand Reinhold Co., New York, 1993.

\bibitem[DFdG20]{corelg}
H.~Dietrich, P.~Faccin, and W.~de~Graaf, \emph{{CoReLG}, computing with real
  {L}ie algebras, {V}ersion 1.54}, \href
  {https://gap-packages.github.io/corelg/}
  {\texttt{https://gap-packages.github.io/}\discretionary
  {}{}{}\texttt{corelg/}}, Jan 2020, Refereed GAP package.

\bibitem[dGGT19a]{quagroup}
W.~A. de~Graaf and T.~GAP~Team, \emph{{QuaGroup}, computations with quantum
  groups, {V}ersion 1.8.2}, \href {https://gap-packages.github.io/quagroup/}
  {\texttt{https://gap-packages.github.io/}\discretionary
  {}{}{}\texttt{quagroup/}}, Oct 2019, Refereed GAP package.

\bibitem[dGGT19b]{sla}
\bysame, \emph{{SLA}, computing with simple {L}ie algebras, {V}ersion 1.5.3},
  {\texttt{https://gap-packages.github.io/}\discretionary {}{}{}\texttt{sla/}},
  Nov 2019, Refereed GAP package.

\bibitem[DK15]{dekem3}
Harm Derksen and Gregor Kemper, \emph{Computational invariant theory}, enlarged
  ed., Encyclopaedia of Mathematical Sciences, vol. 130, Springer, Heidelberg,
  2015, With two appendices by Vladimir L. Popov, and an addendum by Norbert
  A'Campo and Popov, Invariant Theory and Algebraic Transformation Groups,
  VIII.

\bibitem[EG93]{elashgre}
Alexander~G. Elashvili and G\'erard Gr\'elaud, \emph{Classification des
  \'el\'ements nilpotents compacts des alg\`ebres de {L}ie simples}, C. R.
  Acad. Sci. Paris, S\'erie I \textbf{317} (1993), 445--447.

\bibitem[GAP21]{GAP4}
The GAP~Group, \emph{{GAP -- Groups, Algorithms, and Programming, Version
  4.11.1}}, 2021.

\bibitem[GK19]{fplsa}
V.~Gerdt and V.~Kornyak, \emph{{FPLSA}, finitely presented {L}ie algebras,
  {V}ersion 1.2.4}, Jan 2019, Refereed GAP package.

\bibitem[Gra13]{slapaper}
Willem A.~de Graaf, \emph{Computations with nilpotent orbits in {S}{L}{A}},
  \verb+https://arxiv.org/abs/1301.1149+, 2013.

\bibitem[Gra17]{gra16}
\bysame, \emph{Computation with linear algebraic groups}, Monographs and
  Research Notes in Mathematics, CRC Press, Boca Raton, FL, 2017.

\bibitem[GV78]{gavin}
V.~Gatti and E.~Viniberghi, \emph{Spinors of {$13$}-dimensional space}, Adv. in
  Math. \textbf{30} (1978), no.~2, 137--155.

\bibitem[GVY12]{gravinya}
W.A.~de Graaf, {\`E}.B. Vinberg, and O.S. Yakimova, \emph{An effective method
  to compute closure ordering for nilpotent orbits of
  {$\theta$}-representations}, J. Algebra \textbf{371} (2012), 38--62.

\bibitem[Hes79]{hesselink2}
Wim~H. Hesselink, \emph{Desingularizations of varieties of nullforms}, Invent.
  Math. \textbf{55} (1979), no.~2, 141--163.

\bibitem[Hum78]{hum}
James~E. Humphreys, \emph{Introduction to {L}ie algebras and representation
  theory}, Graduate Texts in Mathematics, vol.~9, Springer-Verlag, New
  York-Berlin, 1978, Second printing, revised.

\bibitem[Igu70]{igusa}
Jun-ichi Igusa, \emph{A classification of spinors up to dimension twelve},
  Amer. J. Math. \textbf{92} (1970), 997--1028.

\bibitem[Kac80]{kacnilp}
V.~G. Kac, \emph{Some remarks on nilpotent orbits}, J. Algebra \textbf{64}
  (1980), 190--213.

\bibitem[Kra84]{kraft}
Hanspeter Kraft, \emph{Geometrische {M}ethoden in der {I}nvariantentheorie},
  Aspects of Mathematics, D1, Friedr. Vieweg \& Sohn, Braunschweig, 1984.

\bibitem[LS79]{lusp}
G.~Lusztig and N.~Spaltenstein, \emph{Induced unipotent classes}, J. London
  Math. Soc. (2) \textbf{19} (1979), no.~1, 41--52.

\bibitem[Pan04]{panyushev4}
Dmitri~I. Panyushev, \emph{On reachable elements and the boundary of nilpotent
  orbits in simple {L}ie algebras}, Bull. Sci. math. \textbf{128} (2004),
  859--870.

\bibitem[Pop03]{popov}
V.~L. Popov, \emph{The cone of {H}ilbert null forms}, Tr. Mat. Inst. Steklova
  \textbf{241} (2003), no.~Teor. Chisel, Algebra i Algebr. Geom., 192--209,
  English translation in: {\em Proc. Steklov Inst. Math.} 241 (2003), no. 1,
  177--194.

\bibitem[PS18]{preste}
Alexander Premet and David~I. Stewart, \emph{Rigid orbits and sheets in
  reductive {L}ie algebras over fields of prime characteristic}, J. Inst. Math.
  Jussieu \textbf{17} (2018), no.~3, 583--613.

\bibitem[PT14]{pretop}
Alexander Premet and Lewis Topley, \emph{Derived subalgebras of centralisers
  and finite {$W$}-algebras}, Compos. Math. \textbf{150} (2014), no.~9,
  1485--1548.

\bibitem[SGT18]{sophus}
C.~Schneider and T.~GAP~Team, \emph{{Sophus}, computing in nilpotent {L}ie
  algebras, {V}ersion 1.24}, \href {https://gap-packages.github.io/sophus/}
  {\texttt{https://gap-packages.github.io/}\discretionary
  {}{}{}\texttt{sophus/}}, Apr 2018, Refereed GAP package.

\bibitem[VLE18]{liepring}
M.~Vaughan-Lee and B.~Eick, \emph{{LiePRing}, database and algorithms for {L}ie
  p-rings, {V}ersion 1.9.2}, \href {https://gap-packages.github.io/liepring/}
  {\texttt{https://gap-packages.github.io/}\discretionary
  {}{}{}\texttt{liepring/}}, Oct 2018, Refereed GAP package.

\bibitem[VP89]{povin}
{\`E}.~B. Vinberg and V.~L. Popov, \emph{Invariant theory}, Algebraic geometry,
  4 ({R}ussian), Itogi Nauki i Tekhniki, Akad. Nauk SSSR Vsesoyuz. Inst.
  Nauchn. i Tekhn. Inform., Moscow, 1989, English translation in: V. L. Popov
  and {\`E}. B. Vinberg, {\em Invariant Theory}, in: {\em Algebraic Geometry
  IV}, Encyclopedia of Mathematical Sciences, Vol. 55, Springer-Verlag, {\em
  Proc. Steklov Inst. Math.} 264 (2009), no. 1, 146--158, pp.~137--314.

\bibitem[Yak10]{yakimova2}
Oksana~S. Yakimova, \emph{On the derived algebra of a centraliser}, Bull. Sci.
  Math. \textbf{134} (2010), 579--587.

\end{thebibliography}

\newcommand{\etalchar}[1]{$^{#1}$}
\def\cprime{$'$} \def\cprime{$'$} \def\Dbar{\leavevmode\lower.6ex\hbox to
  0pt{\hskip-.23ex \accent"16\hss}D} \def\cprime{$'$} \def\cprime{$'$}
  \def\cprime{$'$} \def\cprime{$'$} \def\cprime{$'$} \def\cprime{$'$}
\providecommand{\bysame}{\leavevmode\hbox to3em{\hrulefill}\thinspace}
\providecommand{\MR}{\relax\ifhmode\unskip\space\fi MR }
\providecommand{\MRhref}[2]{%
  \href{http://www.ams.org/mathscinet-getitem?mr=#1}{#2}
}
\providecommand{\href}[2]{#2}

\end{document}